\documentclass[12pt]{article}

\usepackage[T1]{fontenc}
\usepackage[utf8]{inputenc}
\usepackage{lmodern}
\usepackage{amsmath,amssymb,amsthm,mathtools}
\usepackage{mathrsfs}
\usepackage{microtype}
\usepackage[a4paper,margin=1.1in]{geometry}
\usepackage{enumitem}
\usepackage{hyperref}

\hypersetup{
	colorlinks=true,
	linkcolor=blue,
	citecolor=blue,
	urlcolor=blue
}

\theoremstyle{plain}
\newtheorem{theorem}{Theorem}[section]
\newtheorem{proposition}[theorem]{Proposition}
\newtheorem{lemma}[theorem]{Lemma}
\newtheorem{corollary}[theorem]{Corollary}

\theoremstyle{definition}
\newtheorem{definition}[theorem]{Definition}

\newtheorem{example}[theorem]{Example}
\newtheorem{remark}[theorem]{Remark}

\newcommand{\F}{\mathbb F}
\newcommand{\R}{\mathbb R}
\newcommand{\C}{\mathbb C}
\newcommand{\Hh}{\mathbb H}
\newcommand{\Oct}{\mathbb O}
\newcommand{\PP}{\mathbb P}
\newcommand{\RP}{\mathbb{RP}}
\newcommand{\OO}{\mathcal O}
\newcommand{\GL}{\operatorname{GL}}
\newcommand{\SO}{\operatorname{SO}}
\newcommand{\Sp}{\operatorname{Sp}}
\newcommand{\KO}{\operatorname{KO}}

\newcommand{\charac}{\operatorname{char}}
\newcommand{\Nrd}{\operatorname{Nrd}}

\title{Coordinate-Deletion Bundles from Composition Algebras: Hopf Defects, $KO$-Classes, and Real Projective Space}
\author{M. Palaisti}
\date{August 2026}

\begin{document}
	\maketitle
	
\begin{abstract}
	Let \(D\in\{\C,\Hh,\Oct\}\) be a real composition division algebra of dimension \(d\in\{2,4,8\}\). We construct an explicit rank-\(d\) real vector bundle $	E_D\longrightarrow\RP^d$ by deleting one homogeneous coordinate on each standard affine chart, interpreting the remaining coordinates as an element of \(D\), and using the corresponding left-multiplication matrices as transition data. The bundle admits a matrix-determined trivialization away from the \(d+1\) coordinate points. Relative to this trivialization, each deleted point has local clutching map $
	S^{d-1}\longrightarrow\SO(d),\; u\longmapsto L_u$, which is respectively the complex, quaternionic, or octonionic Hopf
	clutching map.
	
	A section arising from the same matrices has exactly the coordinate points as nondegenerate zeros. Consequently, $w(E_D)=1+x^d$, and real \(K\)-theory together with Euler-class cancellation gives $E_D\cong\gamma^{\oplus d}$. Thus the construction supplies explicit framed-defect realizations of these familiar bundles. In particular, it gives a semialgebraic octonionic realization of \(\gamma^{\oplus8}\) on \(\RP^8\) with nine specified local Hopf defects. We also describe the associated Pfister quadratic bundle over arbitrary fields.
\end{abstract}
	
	\section{Introduction}

The dimensions $2$, $4$, and $8$ occur simultaneously in the theory of
composition algebras, multiplicative quadratic forms, Hopf bundles, and
real $K$-theory. A real composition division algebra $D$ has a
positive-definite quadratic norm $N$ satisfying $N(uv)=N(u)N(v)$. By Hurwitz's theorem, apart from $\R$ the possibilities are
$\C,\Hh,\Oct$, of dimensions $2,4,8$ respectively. Their norm forms are the real one-, two-, and three-fold Pfister forms. We use
Springer--Veldkamp and Baez for composition algebras and octonions, and
Lam and Knus--Merkurjev--Rost--Tignol for the quadratic-form and
quaternionic background \cite{SpringerVeldkamp,Baez,Lam,KMRT}.

The present construction also grows out of the embeddability viewpoint developed
in \cite{OmanovicThesis}. That work studies how comparatively
small algebraic data (including Pfister forms, quaternionic structures, special subforms, and embeddings associated with algebras with involution) can encode structural information about larger algebraic or geometric objects. The present paper develops a complementary direction of that philosophy. Rather than using an embedding to recover information about an ambient object, we use the multiplication operators supplied by normed composition algebras as explicit
geometric gluing data. In the quaternionic case this is directly rooted in the
Pfister--quaternionic framework of \cite{OmanovicThesis}; the extension to
$\C,\Hh,\Oct$ shows that the same geometric mechanism belongs more broadly to
the topology of normed composition algebras.

For a unit element $u\in D$, left multiplication
$L_u:D\to D$ is orientation-preserving and orthogonal. In the standard
complex and quaternionic models, the maps
\[
S^1\to\SO(2),\qquad S^3\to\SO(4),\qquad u\mapsto L_u,
\]
are the real clutching maps of the corresponding Hopf line bundles. In
the octonionic case
\[
S^7\to\SO(8),\qquad u\mapsto L_u,
\]
is the standard linear clutching map associated with the octonionic
Hopf fibration $S^7\to S^{15}\to S^8$. Because the unit octonions are
not a group, this last statement is a vector-bundle clutching statement,
not a principal-$S^7$ statement. The relation with Clifford modules and
real $K$-theory is classical \cite{ABS,Baez}.

We place these linear multiplication maps into a single construction on
real projective space. A point of $\RP^d$ has $d+1$ homogeneous
coordinates. On the $i$th standard chart we delete the $i$th
coordinate; the remaining $d$ affine coordinates determine an element
$q_i\in D$. Left multiplication by $q_i$ gives a local matrix $M_i$,
and on overlaps we set
\[
g_{ij}=M_iM_j^{-1}.
\]
The cocycle identity occurs in the associative matrix algebra
$M_d(\R)$, so the construction remains valid for $D=\Oct$. What fails
for octonions is only the associative simplification
$L_aL_b^{-1}=L_{ab^{-1}}$.

The construction is explicit relative to fixed data: the standard
projective cover, the ordering of homogeneous coordinates, and an
ordered orthonormal basis of $D$. With those choices fixed, the same
matrices give a specified trivialization of $E_D$ away from the $d+1$
coordinate points. Relative to that trivialization each point has the
local model $u\mapsto L_u$. The local statement should not be confused
with an additive formula for the global bundle class: changing boundary
orientations or coordinate order can change the sign of a clutching
class, and the $d+1$ boundary components do not come with a canonical
sum in $\pi_{d-1}(\SO(d))$. Our global accounting is instead provided
by characteristic classes and $KO$-theory.

Fixing $v\in D^\times$, the local sections $s_i=q_iv$ glue to a global
section with exactly the coordinate points as nondegenerate zeros.
Since $d+1$ is odd,
\[
w_d(E_D)=x^d\ne0.
\]
The complement trivialization forces the lower Stiefel--Whitney classes
to vanish, hence
\[
w(E_D)=1+x^d.
\]
A projective hyperplane can be chosen disjoint from every coordinate
point, so $E_D$ restricts trivially to a copy of $\RP^{d-1}$. Adams'
calculation of $\widetilde{\KO}(\RP^d)$ then gives
\[
[E_D]-d=d([\gamma]-1).
\]
This first determines the stable class. Because $d$ is even, the
stable identification can be sharpened: both $E_D$ and
$\gamma^{\oplus d}$ are oriented rank-$d$ bundles over the
$d$-dimensional CW complex $\RP^d$, and both have the same nonzero Euler
class. An Euler-class cancellation theorem therefore gives
\[
E_D\cong\gamma^{\oplus d}
\qquad(d=2,4,8).
\]
For the cancellation step we use the even-dimensional top-rank theorem
proved in Appendix~A of Gollinger \cite{Gollinger}; see also Malyi's
general study of Euler classes in a fixed stable class \cite{Malyi}.

The octonionic specialization is consequently stronger than a formal
extension of the quaternionic formulas. It gives a concrete
semialgebraic cocycle for the familiar bundle $\gamma^{\oplus8}$ on
$\RP^8$ together with a matrix-determined trivialization off nine
points for which every local defect is the octonionic Hopf map. In the
classical description
\[
\pi_7(\SO(8))\cong\mathbb Z\oplus\mathbb Z,
\]
the map $u\mapsto L_u$ is the generator usually denoted $\sigma$; its
stabilization generates $\pi_7(\SO)\cong\mathbb Z$
\cite{CardimLamMelloRandall}. Thus the local defects retain genuinely
unstable $\SO(8)$ information even though the global bundle is
abstractly $\gamma^{\oplus8}$.

Over a general field the topological construction has a precise
algebraic limitation. Division controls $F$-rational points, but after
scalar extension the norm becomes isotropic, so the inverse matrices
need not be regular on full scheme-theoretic overlaps. On the natural
norm-nonzero open locus the matrix cocycle is a \v{C}ech coboundary. A
natural global quadratic bundle that retains the norm form is instead
\[
\bigl(D\otimes_F\OO(-1),\mathscr N_D\bigr).
\]

For characteristic classes we use Milnor--Stasheff
\cite{MilnorStasheff}; for vector bundles, stable range, and classifying
spaces we use Hatcher and Husemoller \cite{Hatcher,Husemoller}; and for
the $KO$-groups of real projective spaces we use Adams \cite{Adams}.
\subsection*{Main results}

Let \(D\in\{\C,\Hh,\Oct\}\) be a real composition division algebra of
dimension \(d\in\{2,4,8\}\), equipped with a fixed ordered orthonormal
basis.  We obtain the following results.

\begin{description}
	
	\item[\textbf{Coordinate-deletion cocycle.}]
	On the standard affine chart \(U_i\subset\RP^d\), let \(q_i\) be obtained
	by deleting the \(i\)th homogeneous coordinate and normalizing by that
	coordinate.  If \(M(q_i)\) denotes the matrix of left multiplication by
	\(q_i\), then
	\[
	g_{ij}=M(q_i)M(q_j)^{-1}
	\]
	defines a continuous semialgebraic \(\GL_d(\R)\)-valued \v{C}ech cocycle
	on the standard affine cover.  Hence these matrices define a rank-\(d\)
	real vector bundle
	\[
	E_D\longrightarrow\RP^d.
	\]
	
	\item[\textbf{Hopf defects.}]
	Let \(P_0,\ldots,P_d\) be the coordinate points and put
	\[
	X=\RP^d\setminus\{P_0,\ldots,P_d\}.
	\]
	The bundle \(E_D|_X\) has an explicit trivialization represented on
	\(U_i\cap X\) by \(M(q_i)^{-1}\).  Relative to this trivialization, the
	boundary of a small coordinate ball around each \(P_i\) has clutching map
	\[
	S^{d-1}\longrightarrow\SO(d),\qquad u\longmapsto L_u.
	\]
	For \(D=\C,\Hh,\Oct\), this is respectively the complex, quaternionic,
	or octonionic Hopf clutching map.
	
	\item[\textbf{Characteristic and bundle classes.}]
	Writing \(x=w_1(\gamma)\in H^1(\RP^d;\F_2)\), one has
	\[
	w(E_D)=1+x^d
	\]
	and
	\[
	[E_D]-d=d\bigl([\gamma]-1\bigr)
	\qquad\text{in }\widetilde{\KO}(\RP^d).
	\]
	Moreover,
	\[
	E_D\cong\gamma^{\oplus d}.
	\]
	In particular, the octonionic construction gives an explicit
	semialgebraic realization of \(\gamma^{\oplus8}\) on \(\RP^8\), trivial
	away from nine coordinate points whose local clutching maps are
	octonionic Hopf maps.
	
\end{description}

\section{Composition norms and multiplication operators}
\label{sec:composition}

We begin with the real case. Let $D$ be a finite-dimensional unital
real composition division algebra with quadratic norm $N:D\longrightarrow\R$
satisfying
\[
N(uv)=N(u)N(v),\qquad N(1)=1.
\]
The polar form
\[
\langle u,v\rangle
=\frac12\bigl(N(u+v)-N(u)-N(v)\bigr)
\]
is positive definite. By Hurwitz's theorem, the possible dimensions
are $1,2,4,8$. We restrict to
\[
d=\dim_\R D\in\{2,4,8\}.
\]
Thus $D$ is isomorphic to $\C$, $\Hh$, or $\Oct$.

Fix an ordered orthonormal basis $\mathcal B=(e_0,e_1,\ldots,e_{d-1})$ of $D$. The construction below is canonical relative to the standard
projective cover, the ordering of homogeneous coordinates, and this
chosen ordered basis. We do not claim basis-independence of the
resulting explicit cocycle. The characteristic-class and stable-class
conclusions established below hold for every such choice.

For $u\in D$, let
\[
L_u:D\longrightarrow D,
\qquad L_u(v)=uv,
\]
and let $M(u)$ denote the matrix of $L_u$ in the basis $\mathcal B$.
Since multiplication in $D$ is bilinear, the entries of $M(u)$ depend
linearly on the coordinates of $u$.

\begin{lemma}
	\label{lem:similarity}
	For every $u\in D$,
	\[
	\langle L_u v,L_u w\rangle
	=N(u)\langle v,w\rangle.
	\]
	Equivalently,
	\[
	M(u)^T M(u)=N(u)I_d.
	\]
	If $u\ne0$, then
	\[
	M(u)^{-1}=\frac{1}{N(u)}M(u)^T.
	\]
	Moreover, for $d\in\{2,4,8\}$,
	\[
	\det M(u)=N(u)^{d/2}>0
	\qquad (u\ne0).
	\]
\end{lemma}

\begin{proof}
	The norm identity gives $N(u(v+w))=N(u)N(v+w)$, and similarly for $v$ and $w$ separately. Polarizing yields $	\langle uv,uw\rangle=N(u)\langle v,w\rangle$. In an orthonormal basis this is precisely $M(u)^TM(u)=N(u)I_d$. If $u\ne0$, then $N(u)>0$, so the inverse formula
	follows.
	
	Taking determinants gives $\det(M(u))^2=N(u)^d$. Since $D\setminus\{0\}\cong\R^d\setminus\{0\}$ is connected for
	$d\ge2$, the sign of $\det M(u)$ is constant there. At $u=1$ the
	matrix is $I_d$, so the sign is positive.
\end{proof}

\begin{corollary}
	\label{cor:unit-orthogonal}
	If $N(u)=1$, then $L_u\in\SO(d)$.
\end{corollary}

Note that, for $D=\Oct$, compositions of left multiplication operators are
ordinary compositions of real linear maps and therefore associative,
even though multiplication in $\Oct$ is not. In particular, one must
distinguish $L_aL_b$ from $L_{ab}$ in general, but there is no ambiguity in products of the matrices $M(a)$ and $M(b)$.

\section{The coordinate-deletion cocycle}
\label{sec:cocycle}

Let $\RP^d=\{[\alpha_0:\cdots:\alpha_d]\}$ and let
\[
U_i=\{[\alpha_0:\cdots:\alpha_d]:\alpha_i\ne0\},
\qquad i=0,\ldots,d,
\]
be the standard affine cover. For each $i$, let
$\sigma_i:\{0,\ldots,d-1\}
\longrightarrow\{0,\ldots,d\}\setminus\{i\}$ be the increasing bijection. Define $q_i:U_i\longrightarrow D$
by
\[
q_i([\alpha])
=\sum_{r=0}^{d-1}
\frac{\alpha_{\sigma_i(r)}}{\alpha_i}e_r.
\]
Thus $q_i$ is obtained by deleting the $i$th homogeneous coordinate,
normalizing by $\alpha_i$, and reading the remaining $d$ affine
coordinates in the fixed ordered basis of $D$. Let $P_i=[0:\cdots:0:1:0:\cdots:0]$ be the $i$th coordinate point.

\begin{lemma}
	\label{lem:q-zero-general}
	For $p\in U_i$,
	\[
	q_i(p)=0\quad\Longleftrightarrow\quad p=P_i.
	\]
	Consequently, on $U_i\cap U_j$ both $q_i$ and $q_j$ are nonzero.
\end{lemma}

\begin{proof}
	The coefficients of $q_i$ are exactly the $d$ affine coordinate ratios
	other than the $i$th one. They vanish simultaneously exactly at
	$P_i$. If $j\ne i$, then $P_i\notin U_j$, so no coordinate point lies
	in a pairwise overlap involving two distinct charts.
\end{proof}

Put $M_i(p)=M(q_i(p))$. The matrix $M_i$ vanishes at $P_i$, but by
Lemma~\ref{lem:q-zero-general} it is invertible wherever it is used in
a transition function. For $p\in U_i\cap U_j$, define the general cocycle by $	g_{ij}(p)=M_i(p)M_j(p)^{-1}$. The following result highlights the properties of this construction. 

\begin{proposition}
	\label{prop:general-cocycle}
	The maps
	\[
	g_{ij}:U_i\cap U_j\longrightarrow\GL_d(\R)
	\]
	are continuous and semialgebraic and satisfy
	\[
	g_{ii}=I_d,
	\qquad
	g_{ij}=g_{ji}^{-1},
	\qquad
	g_{ij}g_{jk}g_{ki}=I_d.
	\]
	Hence they define a rank-$d$ real vector bundle
	\[
	E_D\longrightarrow\RP^d.
	\]
\end{proposition}

\begin{proof}
	The entries of $M_i$ are linear in the affine coordinates on $U_i$.
	By Lemma~\ref{lem:similarity}, on an overlap
	\[
	M_j^{-1}=\frac{1}{N(q_j)}M_j^T,
	\]
	where the denominator is positive and nonzero. Thus the entries of
	$g_{ij}$ are rational functions with nonvanishing denominator, so the
	maps are continuous and semialgebraic.
	
	The first two cocycle identities are immediate. On a triple overlap,
	\[
	\begin{aligned}
		g_{ij}g_{jk}g_{ki}
		&=(M_iM_j^{-1})(M_jM_k^{-1})(M_kM_i^{-1})\\
		&=I_d.
	\end{aligned}
	\]
	This calculation takes place in the associative matrix algebra
	$M_d(\R)$ and therefore applies without change when $D=\Oct$.
\end{proof}

By Lemma~\ref{lem:similarity}, every $M_i$ is a positive scalar times
an orientation-preserving orthogonal map away from $P_i$. Hence every
$g_{ij}$ has positive determinant and $E_D$ is orientable. Equivalently,
its structure group reduces to $\SO(d)$. This statement is an abstract
reduction of structure group; it does not require replacing the given
transition functions by a pointwise normalization on the same cover.

\begin{example}[The quaternionic matrix]
	\label{ex:quaternion-matrix}
	For $D=\Hh$ with standard basis $1,i,j,k$, write
	\[
	u=x_0+x_1i+x_2j+x_3k.
	\]
	Then
	\[
	M(u)=
	\begin{pmatrix}
		x_0 & -x_1 & -x_2 & -x_3\\
		x_1 & x_0 & -x_3 & x_2\\
		x_2 & x_3 & x_0 & -x_1\\
		x_3 & -x_2 & x_1 & x_0
	\end{pmatrix}.
	\]
	On $\RP^4$,
	\[
	q_0=
	\frac{\alpha_1}{\alpha_0}
	+\frac{\alpha_2}{\alpha_0}i
	+\frac{\alpha_3}{\alpha_0}j
	+\frac{\alpha_4}{\alpha_0}k,
	\]
	while
	\[
	q_1=
	\frac{\alpha_0}{\alpha_1}
	+\frac{\alpha_2}{\alpha_1}i
	+\frac{\alpha_3}{\alpha_1}j
	+\frac{\alpha_4}{\alpha_1}k,
	\]
	and the transition map on $U_0\cap U_1$ is
	\[
	g_{01}=M(q_0)M(q_1)^{-1}.
	\]
	Because $\Hh$ is associative this can also be written
	$g_{01}=L_{q_0q_1^{-1}}$. The analogous simplification is not used in
	the octonionic case.
\end{example}

\section{The global section and Hopf defects}
\label{sec:defects}

Fix a nonzero element $v\in D$. On $U_i$, define
\[ s_i(p)=M_i(p)v=q_i(p)v.
\]

\begin{lemma}
	\label{lem:general-section}
	The local sections $s_i$ glue to a global section
	\[
	s_v:\RP^d\longrightarrow E_D.
	\]
	Its zero set is
	\[
	Z(s_v)=\{P_0,\ldots,P_d\},
	\]
	and every zero is nondegenerate.
\end{lemma}

\begin{proof}
	On $U_i\cap U_j$,
	\[
	g_{ij}s_j
	=M_iM_j^{-1}M_jv
	=M_iv=s_i,
	\]
	so the local sections glue. Because $D$ is a division algebra and $v\ne0$, right multiplication by $v$ is an invertible real linear map. Hence $q_i(p)v=0$ if and only if $q_i(p)=0$, which by Lemma~\ref{lem:q-zero-general} occurs exactly at
	$P_i$.
	
	Use the $d$ affine coordinate ratios on $U_i$ as local coordinates
	centered at $P_i$. In those coordinates the map $p\mapsto q_i(p)$ is
	a linear identification with $D$, and the local representative of the
	section is
	\[
	q\longmapsto qv.
	\]
	Its derivative is the invertible right multiplication map $R_v$.
	Thus every zero is nondegenerate.
\end{proof}

The same matrices give an explicit trivialization away from the zeros.
This is the feature that turns the zero count into a local-defect
description.

\begin{proposition}
	\label{prop:punctured-trivialization}
	Put
	\[
	X=\RP^d\setminus\{P_0,\ldots,P_d\}.
	\]
	Then $E_D|_X$ is trivial. More precisely, if a vector in the fibre is
	represented on $U_i\cap X$ by $w_i\in\R^d$, then
	\[
	\tau_i(w_i)=M_i^{-1}w_i
	\]
	defines a global trivialization
	\[
	\tau:E_D|_X\xrightarrow{\ \cong\ }X\times\R^d.
	\]
\end{proposition}

\begin{proof}
	On $U_i\cap U_j\cap X$, representatives satisfy
	\[
	w_i=g_{ij}w_j=M_iM_j^{-1}w_j.
	\]
	Therefore
	\[
	M_i^{-1}w_i=M_j^{-1}w_j.
	\]
	Since $M_i$ is invertible on $U_i\cap X$, the formulas define a global
	bundle isomorphism.
\end{proof}

We now identify the local clutching data determined by this
trivialization. Choose pairwise disjoint closed coordinate balls $B_i\subset U_i $ centered at $P_i$, small enough that the affine-coordinate map
$p\mapsto q_i(p)$ identifies $B_i$ with a closed ball in $D$. Use the
original $U_i$-frame on $B_i$ and the trivialization
of Proposition~\ref{prop:punctured-trivialization} outside the centers.

\begin{theorem}
	\label{thm:local-hopf}
	On $\partial B_i\cong S^{d-1}$, the clutching map between the local
	$U_i$-frame and the punctured-space trivialization is represented by
	$M_i=L_{q_i}$. After radial normalization it is homotopic through
	$\GL_d^+(\R)$ to
	\[
	\lambda_D:S^{d-1}\longrightarrow\SO(d),
	\qquad
	u\longmapsto L_u,
	\]
	where $N(u)=1$.
	
	For $D=\C$ and $D=\Hh$, this is the underlying real clutching map of
	the complex and quaternionic Hopf line bundle. For $D=\Oct$, it is
	the standard octonionic Hopf clutching map associated with the fibration
	$S^7\to S^{15}\to S^8$; this is a clutching description of a real
	rank-eight vector bundle and does not make $S^7$ into a structure group.
\end{theorem}

\begin{proof}
	Let $w$ denote coordinates in the punctured-space trivialization and
	$w_i$ coordinates in the local $U_i$-frame. By definition of
	$\tau_i$, we have $w=M_i^{-1}w_i$, so $w_i=M_iw$. Thus the transition from the exterior trivialization to the local frame
	is $M_i=L_{q_i}$.
	
	On a sphere $N(q_i)=\varepsilon^2$ around the origin, write
	$q_i=\varepsilon u$ with $N(u)=1$. Bilinearity of multiplication gives
	\[
	L_{q_i}=\varepsilon L_u.
	\]
	The positive scalar factor $\varepsilon$ is homotopically irrelevant in
	$\GL_d^+(\R)$, while Corollary~\ref{cor:unit-orthogonal} gives
	$L_u\in\SO(d)$. For $D=\C,\Hh$ the identification with the real clutching maps of the
	Hopf line bundles is classical. In the octonionic case the same linear
	map is the standard clutching map associated with the octonionic Hopf
	fibration. See \cite{ABS,Baez}.
\end{proof}

We would like to point out that Theorem~\ref{thm:local-hopf} is relative to the ordered projective
coordinates, the standard cover, the ordered basis of $D$, and the
trivialization defined by the matrices $M_i^{-1}$. This is the precise
sense in which the local model is specified. Reordering the affine
coordinates or changing an orientation convention on
$\partial B_i\cong S^{d-1}$ may precompose the displayed map by a
degree-$-1$ self-map of the sphere, and hence may replace its oriented
homotopy class by its negative. No absolute sign for a defect is claimed
without these conventions.

Further, the $d+1$ local clutching maps should not be read as an unproved formula
saying that the global bundle class is the sum of $d+1$ identical Hopf
classes. The complement has several boundary components, their
orientations depend on choices, and $\RP^d$ is nonorientable for the even
$d$ considered here. The local-defect description is geometric; the
rigorous global accounting is the Stiefel--Whitney and $KO$ computation
in Sections~\ref{sec:sw} and~\ref{sec:ko}. This distinction is especially
important in the octonionic case, where the local map has an unstable
class in $\pi_7(\SO(8))$ while the global bundle is analyzed first through
stable $KO$-theory.

\section{Stiefel--Whitney classes}
\label{sec:sw}

Let
\[
x=w_1(\gamma)\in H^1(\RP^d;\F_2)
\]
be the standard generator. Recall
\[
H^*(\RP^d;\F_2)
\cong\F_2[x]/(x^{d+1}).
\]

\begin{theorem}
	\label{thm:top-class}
	For $d\in\{2,4,8\}$,
	\[
	w_d(E_D)=x^d\ne0.
	\]
\end{theorem}

\begin{proof}
	A rank-$d$ real vector bundle over a closed $d$-manifold has top
	Stiefel--Whitney number equal to the mod-$2$ number of zeros of a
	transverse section. By Lemma~\ref{lem:general-section}, the section
	$s_v$ has $d+1$ nondegenerate zeros. Since $d$ is even,
	\[
	d+1\equiv1\pmod2.
	\]
	Therefore
	\[
	\langle w_d(E_D),[\RP^d]_2\rangle=1.
	\]
	The group $H^d(\RP^d;\F_2)$ is generated by $x^d$, so the result follows.
\end{proof}

The punctured-space trivialization also determines all lower
Stiefel--Whitney classes.

\begin{lemma}
	\label{lem:cohomology-injection}
	Let
	\[
	j:X=\RP^d\setminus\{P_0,\ldots,P_d\}\hookrightarrow\RP^d.
	\]
	For every $k<d$, the map
	\[
	j^*:H^k(\RP^d;\F_2)\longrightarrow H^k(X;\F_2)
	\]
	is injective.
\end{lemma}

\begin{proof}
	By excision, the relative cohomology of the pair
	$(\RP^d,X)$ is the direct sum of the local cohomology groups of the
	$d+1$ removed points. Each local pair is equivalent to
	$(D^d,D^d\setminus\{0\})$, whose relative cohomology with $\F_2$
	coefficients is concentrated in degree $d$. Hence
	\[
	H^k(\RP^d,X;\F_2)=0
	\qquad(k<d).
	\]
	The long exact sequence of the pair gives the required injectivity.
\end{proof}

\begin{corollary}[Total Stiefel--Whitney class]
	\label{cor:total-sw}
	For $d\in\{2,4,8\}$,
	\[
	w(E_D)=1+x^d.
	\]
\end{corollary}

\begin{proof}
	By Proposition~\ref{prop:punctured-trivialization}, $E_D|_X$ is trivial,
	so
	\[
	j^*w_k(E_D)=0
	\]
	for every $k>0$. If $k<d$, Lemma~\ref{lem:cohomology-injection}
	forces $w_k(E_D)=0$. The top class is given by
	Theorem~\ref{thm:top-class}.
\end{proof}

\section{The stable class in real $K$-theory}
\label{sec:ko}

The complement trivialization also makes the stable class accessible.
We use Adams' computation of $\widetilde{\KO}(\RP^n)$
\cite{Adams}. Let
\[
\eta=[\gamma]-1\in\widetilde{\KO}(\RP^n).
\]
Adams proved that
\[
\widetilde{\KO}(\RP^n)
\cong\mathbb Z/2^{\varphi(n)},
\]
generated by $\eta$, where $\varphi(n)$ is the number of integers
$s$ satisfying
\[
0<s\le n,
\qquad
s\equiv0,1,2,4\pmod8.
\]
For $d=2,4,8$ one has
\[
2^{\varphi(d)}=2d,
\qquad
2^{\varphi(d-1)}=d.
\]

\begin{lemma}
	\label{lem:generic-hyperplane}
	There exists a projective hyperplane
	\[
	H\cong\RP^{d-1}\subset\RP^d
	\]
	that contains none of the coordinate points. For every such $H$,
	\[
	E_D|_H\cong\varepsilon^d.
	\]
\end{lemma}

\begin{proof}
	For example, take
	\[
	H=\{[\alpha_0:\cdots:\alpha_d]:
	\alpha_0+\cdots+\alpha_d=0\}.
	\]
	At every coordinate point the left-hand side equals $1$, so
	$H\subset X$. Proposition~\ref{prop:punctured-trivialization} therefore
	trivializes $E_D$ on $H$.
\end{proof}

\begin{theorem}
	\label{thm:ko-class}
	For $d\in\{2,4,8\}$,
	\[
	[E_D]-d=d\eta
	\in\widetilde{\KO}(\RP^d).
	\]
	Equivalently,
	\[
	[E_D]-d=[\gamma^{\oplus d}]-d.
	\]
	Thus $E_D$ and $\gamma^{\oplus d}$ are stably isomorphic.
\end{theorem}

\begin{proof}
	Set
	\[
	\delta_D=[E_D]-d\in\widetilde{\KO}(\RP^d).
	\]
	By Lemma~\ref{lem:generic-hyperplane}, $\delta_D$ restricts to zero on
	a projective hyperplane $H\cong\RP^{d-1}$. Any projective hyperplane
	inclusion is equivalent, up to projective automorphism and homotopy, to
	the standard inclusion. Under Adams' description, restriction sends
	the generator $\eta$ on $\RP^d$ to the generator $\eta$ on
	$\RP^{d-1}$. Hence
	\[
	\ker\!\left(
	\widetilde{\KO}(\RP^d)
	\longrightarrow
	\widetilde{\KO}(\RP^{d-1})
	\right)
	\]
	is the order-two subgroup generated by
	\[
	2^{\varphi(d)-1}\eta=d\eta.
	\]
	Thus $\delta_D$ is either $0$ or $d\eta$.
	
	If $\delta_D=0$, then $E_D$ would be stably trivial, and stable
	triviality would force all positive-degree Stiefel--Whitney classes to
	vanish. This contradicts
	$w_d(E_D)=x^d\ne0$ from Theorem~\ref{thm:top-class}. Therefore
	\[
	\delta_D=d\eta.
	\]
	Finally,
	\[
	[\gamma^{\oplus d}]-d=d([\gamma]-1)=d\eta.
	\]
\end{proof}

\begin{remark}
	Equality of the reduced $KO$-classes means that the two bundles become
	isomorphic after adding trivial summands: there exists $r\ge0$ such that
	\[
	E_D\oplus\varepsilon^r
	\cong
	\gamma^{\oplus d}\oplus\varepsilon^r.
	\]
	The $KO$ calculation by itself does not cancel these summands. The
	cancellation is a separate unstable argument given below.
\end{remark}

\begin{remark}
	For $d=2,4,8$, the element $d\eta$ is the unique nonzero element of
	order two in $\widetilde{\KO}(\RP^d)$. The coordinate-deletion bundle
	therefore realizes this order-two class by a cocycle whose support, in
	the sense of Proposition~\ref{prop:punctured-trivialization}, is
	concentrated at $d+1$ explicit point defects.
\end{remark}

\subsection{Euler cancellation in top rank}
\label{sec:cancellation}

We now pass from the stable class to the actual rank-$d$ bundle. We use
the following even-dimensional cancellation principle.

\begin{theorem}
	\label{thm:euler-cancellation}
	Let $n$ be even and let $X$ be a finite CW complex of dimension $n$.
	If $V_0$ and $V_1$ are oriented real rank-$n$ vector bundles over $X$
	that are stably isomorphic and satisfy
	\[
	e(V_0)=e(V_1)\in H^n(X;\mathbb Z),
	\]
	then
	\[
	V_0\cong V_1.
	\]
\end{theorem}

The exact statement above is proved in Appendix~A of Gollinger
\cite{Gollinger}; Malyi gives a broader analysis of the possible Euler
classes within a stable equivalence class \cite{Malyi}. One may also
combine the standard stable-range cancellation theorem for bundles of
rank $>\dim X$ with the one-line stabilization form of the same Euler
cancellation result; see \cite{Husemoller,Gollinger}.

\begin{theorem}
	\label{thm:actual-identification}
	For every $D\in\{\C,\Hh,\Oct\}$, with $d=\dim_\R D$, there is an
	isomorphism of real rank-$d$ vector bundles
	\[
	E_D\cong\gamma^{\oplus d}
	\qquad\text{over }\RP^d.
	\]
	In particular,
	\[
	E_{\Oct}\cong\gamma^{\oplus8}.
	\]
\end{theorem}

\begin{proof}
	The bundles are stably isomorphic by Theorem~\ref{thm:ko-class}. Both bundles are oriented: the transition functions of \(E_D\) have
	positive determinant by Lemma~\ref{lem:similarity}, while
	\[
	w_1(\gamma^{\oplus d})=dx=0
	\]
	because \(d\) is even. Their top Stiefel--Whitney classes agree:
	\[
	w_d(E_D)=x^d
	=w_d(\gamma^{\oplus d}).
	\]
	For an oriented rank-$d$ bundle, $w_d$ is the mod-$2$ reduction of the
	Euler class.For even \(d\), reduction modulo \(2\) induces an isomorphism
	\[
	H^d(\RP^d;\mathbb Z)\longrightarrow H^d(\RP^d;\F_2).
	\]
	Since \(w_d\) is the mod-\(2\) reduction of the Euler class, the common
	nonzero class \(x^d\) implies that both Euler classes are the unique
	nonzero element of \(H^d(\RP^d;\mathbb Z)\cong\mathbb Z/2\). Hence
	\[
	e(E_D)=e(\gamma^{\oplus d}).
	\]
	Theorem~\ref{thm:euler-cancellation} now gives the actual isomorphism.
\end{proof}

\begin{corollary}[Pullback to the sphere]
	\label{cor:sphere-pullback}
	Let $p:S^d\to\RP^d$ be the antipodal double cover. Then
	\[
	p^*E_D\cong\varepsilon^d.
	\]
	The pullback of the section $s_v$ has two nondegenerate zeros above each
	coordinate point. With the standard orientation on $S^d$, the two
	indices in each antipodal pair are opposite, so the signed zero count is
	zero.
\end{corollary}

\begin{proof}
	The pullback of the tautological line bundle is trivial, so the first
	statement follows from Theorem~\ref{thm:actual-identification}. Near the
	two lifts of a coordinate point, the projective affine coordinates are
	related by the antipodal map on $S^d$. Its degree is
	$(-1)^{d+1}=-1$ because $d$ is even, while the local derivative of the
	section is the same invertible right-multiplication map in both affine
	models. Thus the two local indices have opposite signs.
\end{proof}

\section{Complex, quaternionic, and octonionic refinements}
\label{sec:special-cases}

Theorem~\ref{thm:actual-identification} treats all three dimensions at
once. The associative cases have additional structure, while the
octonionic case has a distinctive unstable local class.

\subsection{The complex case}

For $D=\C$, associativity gives
\[
g_{ij}=L_{q_iq_j^{-1}},
\]
and these maps are complex linear. Hence $E_{\C}$ is the underlying
real bundle of a complex line bundle on $\RP^2$. Its first Chern class
is the unique nonzero element of
\[
H^2(\RP^2;\mathbb Z)\cong\mathbb Z/2,
\]
because its mod-$2$ reduction is
$w_2(E_{\C})=x^2$. Thus the complex structure refines the real
isomorphism
\[
E_{\C}\cong\gamma^{\oplus2}
\]
from Theorem~\ref{thm:actual-identification}.

\subsection{The quaternionic case}

For $D=\Hh$, associativity similarly gives
\[
g_{ij}=L_{q_iq_j^{-1}}.
\]
The transition functions lie in $\Hh^\times$, and since
$\Hh^\times/\Sp(1)\cong\R_{>0}$ is contractible, the structure group
reduces to $\Sp(1)$. Quaternionic line bundles over $\RP^4$ are
classified by
\[
[\RP^4,B\Sp(1)]\cong H^4(\RP^4;\mathbb Z)\cong\mathbb Z/2;
\]
indeed the $4$-type of $B\Sp(1)$ is $K(\mathbb Z,4)$. Since
$w_4(E_{\Hh})=x^4\ne0$, the bundle is the unique nontrivial
quaternionic line bundle. Its underlying real bundle is therefore
\[
E_{\Hh}\cong\gamma^{\oplus4},
\]
consistent with Theorem~\ref{thm:actual-identification}. See
\cite{Hatcher,CrowleyGoette} for the classifying-space viewpoint.

This recovers the original two-fold Pfister construction and is the
specialization most directly connected with the Pfister--quaternionic
embeddability perspective of Omanovic's thesis \cite{OmanovicThesis}. The
present argument does not claim that the coordinate-deletion bundle construction
appears in the thesis; rather, the thesis supplies the algebraic viewpoint from
which the quaternionic realization arose. If
\[
A=\left(\frac{-a,-b}{\R}\right),\qquad a,b>0,
\]
then a basis adapted to $A$ changes the coordinate expression of the
norm to $\langle1,a,b,ab\rangle$ but not the real bundle isomorphism
class.

\subsection{The octonionic case}

Let $D=\Oct$ and $d=8$. The unit octonions form a Moufang loop rather
than an associative topological group, so there is no analogue of the
$\Sp(1)$ reduction above. In particular,
\[
L_aL_b^{-1}
\]
need not be a single left multiplication operator. The coordinate-
deletion matrix cocycle, however, remains valid.

\begin{theorem}
	\label{thm:octonionic}
	The coordinate-deletion construction defines a continuous
	semialgebraic oriented rank-eight bundle
	\[
	E_{\Oct}\longrightarrow\RP^8
	\]
	with the following properties:
	\begin{enumerate}
		\item $E_{\Oct}$ is trivial on the complement of the nine coordinate
		points by the matrix-determined trivialization $M_i^{-1}$;
		\item relative to the fixed coordinate and basis conventions, every
		coordinate point has local clutching map
		\[
		\sigma:S^7\longrightarrow\SO(8),\qquad u\longmapsto L_u,
		\]
		the standard octonionic Hopf clutching map associated with
		$S^7\to S^{15}\to S^8$;
		\item
		\[
		w(E_{\Oct})=1+x^8,
		\qquad w_8(E_{\Oct})=x^8\ne0;
		\]
		\item
		\[
		[E_{\Oct}]-8=8\bigl([\gamma]-1\bigr)
		\quad\text{in }\widetilde{\KO}(\RP^8);
		\]
		\item the stable identification cancels, and
		\[
		E_{\Oct}\cong\gamma^{\oplus8};
		\]
		\item the local map $\sigma$ is one of the two classical generators in
		\[
		\pi_7(\SO(8))\cong\mathbb Z\oplus\mathbb Z,
		\]
		and its stabilization is a generator of
		$\pi_7(\SO)\cong\mathbb Z$.
	\end{enumerate}
\end{theorem}

\begin{proof}
	Statements 1.--4. are
	Proposition~\ref{prop:punctured-trivialization},
	Theorem~\ref{thm:local-hopf}, Corollary~\ref{cor:total-sw}, and
	Theorem~\ref{thm:ko-class}. Statement 5. is
	Theorem~\ref{thm:actual-identification}. For 6., the classical
	notation for $\pi_7(\SO(8))\cong\mathbb Z\oplus\mathbb Z$ takes
	$\sigma(u)(v)=uv$ as one generator; the associated sphere bundle is the octonionic Hopf
	fibration. Stabilization sends $\sigma$ to a generator of
	$\pi_7(\SO)\cong\mathbb Z$; see
	\cite{CardimLamMelloRandall}. This is also compatible with the
	Clifford-module and Bott-periodicity interpretation in
	\cite{ABS,Baez}.
\end{proof}

The abstract isomorphism $E_{\Oct}\cong\gamma^{\oplus8}$ does not
identify the matrix-determined punctured-space trivialization with the
standard direct-sum presentation of $\gamma^{\oplus8}$. The new data
are therefore not a new rank-eight isomorphism class, but an explicit
framed-defect realization of that class in which each local boundary
map is the unstable octonionic generator $\sigma\in\pi_7(\SO(8))$.
This is precisely the structure that disappears if one remembers only
the abstract bundle or only its stable $KO$ class.

In addition, the map $u\mapsto L_u$ is naturally tied to one of the eight-dimensional
representations appearing in the triality picture for $\operatorname{Spin}(8)$; left
multiplication, right multiplication, and the vector representation are
interchanged by triality. We use this only as context: the proofs above
require neither a triality identification nor a group structure on the
unit octonions. See \cite{Baez}.

\section{Composition algebras over arbitrary fields}
\label{sec:arbitrary-field}

We now explain what remains algebraic over a general field. Let $F$ be
a field with
\[
\charac F\ne2,
\]
and let $D$ be a unital composition division algebra over $F$ of
dimension
\[
d\in\{2,4,8\}.
\]
Write
\[
N_D:D\longrightarrow F
\]
for its multiplicative quadratic norm. These norms are respectively
one-, two-, and three-fold Pfister forms; see
\cite{Lam,SpringerVeldkamp}. Let $\bar u$ denote the standard
conjugation in the composition algebra. For $N_D(u)\ne0$ one has
\[
u^{-1}=\frac{\bar u}{N_D(u)}
\]
and, for left multiplication,
\[
L_u^{-1}=\frac{1}{N_D(u)}L_{\bar u}.
\]

Fix an ordered $F$-basis $e_0,\ldots,e_{d-1}$ of $D$. On the standard
affine charts $U_i\subset\PP^d_F$, define the coordinate-deletion
elements $q_i$ exactly as in the real case. At the level of
$F$-rational points, division implies that every nonzero $q_i$ is
invertible, so the matrices
\[
g_{ij}=M(q_i)M(q_j)^{-1}
\]
are defined on
\[
U_i(F)\cap U_j(F).
\]
This does not in general define an algebraic vector bundle on all of
$\PP^d_F$.

\begin{proposition}
	\label{prop:scheme-obstruction}
	After scalar extension to an algebraic closure $\overline F$, the norm
	$N_D$ is isotropic. Consequently, on a scheme-theoretic overlap
	$U_i\cap U_j$ the divisor
	\[
	N_D(q_j)=0
	\]
	is in general nonempty after scalar extension, and the entries of
	\[
	M(q_j)^{-1}
	=\frac{1}{N_D(q_j)}M(\overline{q_j})
	\]
	are not regular along that divisor. Thus anisotropy on $F$-rational
	points alone does not make the coordinate-deletion cocycle a regular
	$\GL_d$-valued cocycle on the full standard cover of $\PP^d_F$.
\end{proposition}

\begin{proof}
	Over an algebraically closed field, every nondegenerate quadratic form
	of dimension at least two is isotropic. The scalar extension of
	$N_D$ is therefore isotropic. On $U_j$, the coefficients of $q_j$ are
	independent affine coordinates. The additional overlap condition
	$\alpha_i\ne0$ requires one specified coefficient of $q_j$ to be
	nonzero. Since the nondegenerate isotropic quadric $N_D=0$ is not
	contained in that coordinate hyperplane, it meets this open condition.
	At such geometric points the displayed denominator vanishes, so the
	inverse matrix is not regular there.
\end{proof}

There is nevertheless a natural open locus on which the matrix formulas
are unquestionably algebraic. Put
\[
W_i=\{p\in U_i:N_D(q_i(p))\ne0\}
\]
and
\[
X_D=\bigcup_{i=0}^d W_i\subseteq\PP^d_F.
\]

\begin{proposition}
	\label{prop:algebraic-triviality-general}
	On the cover $\{W_i\}$ of $X_D$, the transition matrices
	\[
	g_{ij}=M_iM_j^{-1}
	\]
	define an algebraic rank-$d$ vector bundle. This bundle is
	algebraically trivial.
\end{proposition}

\begin{proof}
	On $W_i$, the matrix $M_i$ is a regular $\GL_d$-valued map. Hence the
	transition functions are regular on $W_i\cap W_j$. But the cocycle is
	already the \v{C}ech coboundary of the $0$-cochain $\{M_i\}$:
	\[
	g_{ij}=M_iM_j^{-1}.
	\]
	Changing local frames by $M_i^{-1}$ makes every transition matrix the
	identity.
\end{proof}

The contrast with the real projective-space construction is now
transparent. Over $\R$, the matrices $M_i$ are allowed to be singular
at the coordinate point $P_i$, because $P_i$ lies in no pairwise
overlap involving a second chart. Removing those points makes the
same matrices into a global trivializing $0$-cochain. Thus the
nontrivial real bundle is carried precisely by the local defects at the
chart centers.

\subsection{The global Pfister quadratic bundle}

Although the coordinate-deletion cocycle is algebraically trivial on
$X_D$ and does not generally extend as a regular cocycle to the whole
standard cover, the norm itself has a natural global quadratic realization.
Let
\[
\mathscr E_D=D\otimes_F\OO_{\PP^d}(-1).
\]
This is a rank-$d$ algebraic vector bundle, noncanonically isomorphic to
\[
\OO_{\PP^d}(-1)^{\oplus d}.
\]
Since $N_D$ is homogeneous of degree two, it defines an
$\OO(-2)$-valued quadratic map
\[
\mathscr N_D:\mathscr E_D\longrightarrow\OO_{\PP^d}(-2).
\]
Locally, if $t$ is a generator of $\OO(-1)$, then
\[
\mathscr N_D(u\otimes t)=N_D(u)t^2.
\]
A change $t\mapsto ft$ multiplies the expression by $f^2$, exactly as
required.

\begin{definition}
	The pair
	\[
	(\mathscr E_D,\mathscr N_D)
	\]
	is the \emph{Pfister quadratic bundle} associated with the composition
	algebra $D$.
\end{definition}

\begin{proposition}
	\label{prop:real-quadratic-bundle}
	If $F=\R$ and $D$ is a composition division algebra of dimension
	$d\in\{2,4,8\}$, then the underlying real topological vector bundle of
	$\mathscr E_D$ is
	\[
	\gamma^{\oplus d}.
	\]
	By Theorem~\ref{thm:actual-identification}, it is actually isomorphic
	to the coordinate-deletion bundle $E_D$ for all $d\in\{2,4,8\}$.
\end{proposition}

\begin{proof}
	The real line bundle underlying $\OO_{\PP^d}(-1)$ is the tautological
	line bundle $\gamma$. Tensoring it with the $d$-dimensional real vector
	space $D$ gives $\gamma^{\oplus d}$. Apply
	Theorem~\ref{thm:actual-identification}.
\end{proof}

\subsection{Quaternion Pfister coordinates}

The original quaternionic formula is recovered as follows. Let
\[
A=\left(\frac{-a,-b}{F}\right)
\]
with basis $1,i,j,k$, where
\[
i^2=-a,\qquad j^2=-b,\qquad ij=-ji,
\qquad k=ij.
\]
For
\[
u=x_0+x_1i+x_2j+x_3k,
\]
the reduced norm is the two-fold Pfister form
\[
\Nrd(u)=x_0^2+ax_1^2+bx_2^2+abx_3^2
=\langle1,a,b,ab\rangle(u),
\]
and the left multiplication matrix is
\[
M(u)=
\begin{pmatrix}
	x_0 & -a x_1 & -b x_2 & -abx_3\\
	x_1 & x_0 & -b x_3 & b x_2\\
	x_2 & a x_3 & x_0 & -a x_1\\
	x_3 & -x_2 & x_1 & x_0
\end{pmatrix}.
\]
Moreover,
\[
M(u)^{-1}
=\frac{1}{\Nrd(u)}M(\bar u)
\]
whenever $\Nrd(u)\ne0$. Thus the earlier two-fold Pfister cocycle is
exactly the $d=4$ specialization of the general composition-algebra
construction.

\end{document}